\documentclass[a4paper,12pt]{article}
\usepackage[english]{babel}
\usepackage{tikz}
\usetikzlibrary{matrix,arrows,decorations.markings}
\usepackage{amsmath,amsfonts,amssymb,amsthm,url,textcomp}
\usepackage{csquotes}
\usepackage{a4wide}
\usepackage[numbers]{natbib}
\usepackage{fancyhdr}
\usepackage{anyfontsize}
\usepackage{hyperref}
\usepackage{hhline}
\usepackage{float}

\allowdisplaybreaks

\newtheorem{theorem}{Theorem}\numberwithin{theorem}{section}
\newtheorem{definition}[theorem]{Definition}
\newtheorem{lemma}[theorem]{Lemma}
\newtheorem{corollary}[theorem]{Corollary}
\newtheorem{proposition}[theorem]{Proposition}

\newtheorem{question}[theorem]{Question}
\newtheorem{problem}[theorem]{Problem}
\newtheorem{theoremm}{Theorem}\numberwithin{theoremm}{subsection}
\newtheorem{deffinition}[theoremm]{Definition}

\newtheorem{nottation}[theoremm]{Notation}

\numberwithin{theoremmm}{subsubsection}

\theoremstyle{remark}

\newtheorem{example}[theorem]{Example}

\newcommand{\Aut}{\operatorname{Aut}}
\newcommand{\Alt}{\operatorname{Alt}}

\newcommand{\Aff}{\operatorname{Aff}}

\newcommand{\Sym}{\operatorname{Sym}}

\newcommand{\A}{\operatorname{A}}

\newcommand{\C}{\operatorname{C}}

\newcommand{\id}{\operatorname{id}}
\newcommand{\N}{\operatorname{N}}

\newcommand{\e}{\mathrm{e}}

\newcommand{\J}{\operatorname{J}}

\newcommand{\im}{\operatorname{im}}

\newcommand{\G}{\mathcal{G}}

\newcommand{\D}{\operatorname{D}}

\newcommand{\F}{\operatorname{F}}

\renewcommand{\F}{\mathcal{F}}
\renewcommand{\O}{\operatorname{O}}

\newcommand{\IN}{\mathbb{N}}

\newcommand{\Sp}{\operatorname{Sp}}

\newcommand{\IF}{\mathbb{F}}

\newcommand{\End}{\operatorname{End}}

\newcommand{\IZ}{\mathbb{Z}}

\newcommand{\app}{\operatorname{app}}
\newcommand{\enapp}{\operatorname{endapp}}
\newcommand{\affapp}{\operatorname{affapp}}

\newcommand{\Dic}{\operatorname{Dic}}
\newcommand{\rem}{\operatorname{r}}
\newcommand{\quot}{\operatorname{q}}

\renewcommand{\J}{\operatorname{JK}}
\newcommand{\dist}{\operatorname{dist}}

\newcommand{\Acal}{\mathcal{A}}
\newcommand{\Pcal}{\mathcal{P}}

\begin{document}

\title{Worst-case approximability of functions on finite groups by endomorphisms and affine maps}

\author{Alexander Bors\thanks{Johann Radon Institute for Computational and Applied Mathematics (RICAM), Altenberger Stra{\ss}e 69, 4040 Linz, Austria. \newline E-mail: \href{mailto:alexander.bors@ricam.oeaw.ac.at}{alexander.bors@ricam.oeaw.ac.at} \newline This work was supported by the Austrian Science Fund (FWF) under Grant J4072-N32 \enquote{Affine maps on finite groups}. \newline 2010 \emph{Mathematics Subject Classification}: Primary: 20D60. Secondary: 20F69. \newline \emph{Key words and phrases:} Finite groups, Affine functions, Hamming metric, Nonlinearity.}}

\date{\today}

\maketitle

\abstract{We study the maximum Hamming distance (or rather, the complementary notion of \enquote{minimum approximability}) of a general function on a finite group $G$ to either of the sets $\End(G)$ and $\Aff(G)$, of group endomorphisms of $G$ and affine maps on $G$ respectively, the latter being a certain generalization of endomorphisms. We give general bounds on these two quantities and discuss an infinite class of extremal examples (where each of the two Hamming distances can be made as large as generally possible). Finally, we compute the precise values of the two quantities for all finite groups $G$ with $|G|\leq15$.}

\section{Introduction}\label{sec1}

\subsection{Motivation and main results}\label{subsec1P1}

In this paper, whenever we speak of a function \emph{on} some set $M$, we mean a function $M\rightarrow M$. In order to put the results of this paper into a broader context, consider the following general concept:

\begin{deffinition}\label{approxDef}
Let $M_1,M_2$ be finite sets and $\F$ a set of functions $M_1\rightarrow M_2$. For a function $g:M_1\rightarrow M_2$, we denote by
\[
\app_{\F}(g):=\max_{f\in\F}{|\{x\in M_1\mid f(x)=g(x)\}|}
\]
the \emph{$\F$-approximability of $g$}, and set
\[
\app_{\F}(M_1,M_2):=\min_{g:M_1\rightarrow M_2}{\app_{\F}(g)},
\]
the \emph{minimum} (or \emph{worst-case}) \emph{$\F$-approximability of a function $M_1\rightarrow M_2$}. In case $M_1=M_2=M$, we also write $\app_{\F}(M)$ instead of $\app_{\F}(M,M)$ and call this the \emph{minimum} (or \emph{worst-case}) \emph{$\F$-approximability on $M$}.
\end{deffinition}

Such notions of function approximability (or the complementary notion of the \emph{Hamming distance} between functions) are studied in the literature across several disciplines. We mention the following examples:

\begin{enumerate}
\item Crytography: The special case when $M_1$ and $M_2$ are finite-dimensional vector spaces over a finite field $K$ and $\F$ is the collection of $K$-affine functions $M_1\rightarrow M_2$ is heavily studied in cryptography, due to the need of encryption procedures which resist so-called \emph{linear attacks} (keyword \enquote{nonlinearity}, see also \cite[Introduction]{CT02a}).
\item Coding theory: Let $l$ be a positive integer and $\Acal$ a finite alphabet. Set $M_1:=\{1,\ldots,l\}$ and $M_2:=\Acal$. A code of words of length $l$ over $\Acal$ can be understood as a family of $\F$ of functions $M_1\rightarrow M_2$, and the fundamental problem of finding a code containing many valid code words while being able to correct or at least detect many errors can be formally stated as the problem of finding a \enquote{large} family $\F$ of functions $M_1\rightarrow M_2$ such that for all $f\in\F$, $\app_{\F\setminus\{f\}}(f)\leq l-C$ for a \enquote{large} constant $C$.
\item Group theory: There is a particularly rich literature on the $\Aut(G)$-approximability of power functions on a finite group $G$, particularly the inversion, squaring and cubing function, see \cite[Subsection 1.1]{Bor16b} for an overview. Furthermore, the main result of \cite{LS12a} can be viewed as providing nontrivial upper bounds on the approximability of word maps on nonabelian finite simple groups $S$ by constant functions (here, $M_1=S^d$ with $d$ the number of variables in the word, and $M_2=S$).
\end{enumerate}

This paper is a group-theoretic contribution; its aim is to study $\app_{\F}(G)$ where $G$ is a finite group and $\F$ is either the collection (monoid) $\End(G)$ of endomorphisms of $G$ or the larger collection of so-called \emph{affine maps of $G$}:

\begin{deffinition}\label{affineDef}
Let $G$ be a group, $g\in G$, $\varphi$ an endomorphism of $G$. Then the function $\A_{g,\varphi}:G\rightarrow G,x\mapsto g\varphi(x)$, is called the \emph{(left-)affine map of $G$ with respect to $g$ and $\varphi$}. We denote the set of affine maps on $G$ by $\Aff(G)$.
\end{deffinition}

We note that the notion of an affine map on a group and the notation $\Aff(G)$ already appeared in the author's paper \cite{Bor16a}, where $\Aff(G)$ denoted something different, namely $\{\A_{g,\alpha}\mid g\in G,\alpha\in\Aut(G)\}$, the set of \emph{bijective} affine maps on $G$, which forms a subgroup of the symmetric group on $G$. We also note the earlier paper \cite{JS75a}, which used the terminology \enquote{affine transformation} instead of \enquote{bijective affine map}.

We give two reasons for studying not only $\app_{\End(G)}(G)$, but also $\app_{\Aff(G)}(G)$: Firstly, for many (though not all) finite groups $G$, determining the precise value of $\app_{\End(G)}(G)$ is a relatively easy problem (see, for instance, Proposition \ref{zeroProp} and the paragraph thereafter), whereas $\app_{\Aff(G)}(G)$ is more delicate in general (due to the greater \enquote{freedom in mapping behavior} which affine maps enjoy compared to endomorphisms), thus leading to interesting problems and questions even in cases where $\app_{\End(G)}(G)$ is trivial. Secondly, as mentioned above, $\app_{\Aff(G)}(G)$ has already been heavily studied in the abelian setting by cryptographers, and it is of intrinsic interest to extend these studies to the nonabelian setting. In this context, we would also like to draw the reader's attention to the papers \cite{FX18a,Poi06a,Poi12a,PP11a,Xu15a}, in which other function properties of cryptographic significance (perfect nonlinear, almost perfect nonlinear and bent functions) are studied over nonabelian finite groups.

In the following, we will present our main results in the form of Theorem \ref{mainTheo} below, but before, we introduce some more notation for a more concise formulation:

\begin{nottation}\label{mainNot}
Let $G$ be a finite group, $f$ a function on $G$.

\begin{enumerate}
\item We denote by $\enapp(f):=\app_{\End(G)}(f)$ the \emph{endomorphic approximability of $f$}.
\item We denote by $\affapp(f):=\app_{\Aff(G)}(f)$ the \emph{affine approximability of $f$}.
\item We denote by $\enapp(G):=\app_{\End(G)}(G)=\min_{f:G\rightarrow G}{\enapp(f)}$ the \emph{minimum} (or \emph{worst-case}) \emph{endomorphic approximability on $G$}.
\item We denote by $\affapp(G):=\app_{\Aff(G)}(G)=\min_{f:G\rightarrow G}{\affapp(f)}$ the \emph{minimum} (or \emph{worst-case}) \emph{affine approximability on $G$}.
\end{enumerate}
\end{nottation}

Note that $\enapp(G)\leq\affapp(G)$, as all endomorphisms are affine maps, and that since $\Aff(G)$ contains all constant functions on $G$, we always have $\affapp(G)\geq 1$, whereas there are various examples of finite groups $G$ such that $\enapp(G)=0$ (see Proposition \ref{zeroProp} and the paragraph thereafter).

In this paper, we will show the following bounds and asymptotic results on $\enapp$ and $\affapp$:

\begin{theoremm}\label{mainTheo}
The following hold for all finite groups $G$:

\begin{enumerate}
\item If $G$ is nontrivial, then
\[
0\leq\enapp(G)\leq(\frac{1}{\log{2}}+\frac{1}{\log{|G|}})\log^2{|G|}
\]
and
\[
1\leq\affapp(G)\leq(\frac{1}{\log{2}}+\frac{2}{\log{|G|}})\log^2{|G|}.
\]
In particular, we have
\[
\enapp(G)\leq\affapp(G)=\O(\log^2{|G|})
\]
as $|G|\to\infty$.
\item There is an infinite class of finite groups $H$ with $\enapp(H)\geq\log_2{|H|}$ and $\affapp(H)\geq\log_2{|H|}+1$. In particular, we have
\[
\limsup_{|G|\to\infty}{\enapp(G)}=\limsup_{|G|\to\infty}{\affapp(G)}=\infty.
\]
\item There is an infinite class of finite groups $G$ satisfying $\enapp(G)=0$ and $\affapp(G)=1$. In particular, we have
\[
\liminf_{|G|\to\infty}{\enapp(G)}=0
\]
and
\[
\liminf_{|G|\to\infty}{\affapp(G)}=1.
\]
\end{enumerate}
\end{theoremm}

Theorem \ref{mainTheo}(3) is interesting in view of the aforementioned connections to cryptography, since we will see later (in Corollary \ref{abelianCor}) that in the abelian setting, in which cryptographers usually work, one cannot bring the endomorphic (resp.~affine) approximability of a function on $G$ below $1$ (resp.~$2$), whereas Theorem \ref{mainTheo}(3) asserts that it is possible to do so on suitably chosen (nonabelian) finite groups. Therefore, at least with respect to that particular \enquote{measure of non-affineness}, one can do a little bit better on some nonabelian groups than in the well-studied abelian setting. Note, however, that it is \emph{not} possible to have a single function $f$ on a finite group $G$ such that $\enapp(f)=0$ and $\affapp(f)=1$ simultaneously, since $\affapp(f)=1$ implies that $f$ is a permutation on $G$ (as all constant functions on $G$ are affine), so that such an $f$ would assume the value $1_G$ at some $x\in G$ and would thus agree with the trivial endomorphism of $G$ on the nonempty set $\{x\}$, a contradiction.

Moreover, we note that the infinite class of finite groups which we will give as an example to prove Theorem \ref{mainTheo}(3) also has another interesting property, which was already noted by the author in his unpublished preprint \cite{Bor14a}: These groups are nonabelian with commutative endomorphism monoid. Nonabelian groups with the weaker property of having an abelian automorphism group have been studied by various authors before (see e.g.~the survey \cite{KY18a}), and \enquote{our} groups with commutative endomorphism monoid were first introduced and studied in \cite{JK75a}, where it was shown that their automorphism groups are abelian.

\subsection{Overview of this paper}\label{subsec1P2}

Sections \ref{sec2}--\ref{sec4} of this paper serve to prove the three parts of Theorem \ref{mainTheo}.

In Section \ref{sec2}, we discuss some methods to prove lower bounds on $\enapp(G)$ and $\affapp(G)$, which will allow us to prove Theorem \ref{mainTheo}(2). As a further application, we give theoretical arguments to determine the precise values of $\enapp(G)$ and $\affapp(G)$ for all finite groups $G$ with $|G|\leq 7$, which will be used in the proof of Theorem \ref{mainTheo}(1) in Section \ref{sec3}.

Section \ref{sec3} consists mainly of the proof of Lemma \ref{genAppLem}, which provides some general bounds on the worst-case $\F$-approximability of functions $M_1\rightarrow M_2$ where $M_1$ and $M_2$ are finite sets and $\F$ is a "small" family of functions $M_1\rightarrow M_2$. Theorem \ref{mainTheo}(1) will follow swiftly from this and the case study of groups up to order $7$ from the previous section.

Section \ref{sec4} is dedicated to the proof of Theorem \ref{mainTheo}(3) by a careful examination of the groups already studied by Jonah and Konvisser in \cite{JK75a}.

Section \ref{sec5} provides a comprehensive lists of the values of $\enapp(G)$ and $\affapp(G)$ for all finite groups $G$ with $|G|\leq15$, which were obtained through some computations in GAP \cite{GAP4}.

Finally, in Section \ref{sec6}, we discuss some open problems and questions for further research.

\subsection{Notation and terminology}\label{subsec1P3}

The notation and terminology defined in this subsection will be used throughout the paper without further mentioning; more notation and terminology will be explicitly introduced throughout the text where appropriate.

We denote by $\IN$ the set of natural numbers, including $0$, and by $\IN^+$ the set of positive integers. For a function $f$, the image of $f$ is denoted by $\im(f)$, and the restriction of $f$ to a set $M$ by $f_{\mid M}$.

The exponent of a finite group $G$ is denoted by $\exp(G)$, its center by $\zeta G$ and its derived (or commutator) subgroup by $G'$. $\D_{2n}$ and $\Dic_{4n}$ denote the dihedral group of order $2n$ and the dicyclic group of order $4n$ respectively. The symmetric and alternating groups of degree $n$ are denoted by $\Sym(n)$ and $\Alt(n)$ respectively. The kernel of a group homomorphism $\varphi$ is denoted by $\ker(\varphi)$. For a prime $p$, the finite field with $p$ elements is denoted by $\IF_p$.

As usual, Euler's constant is denoted by $\e$, and for a positive real number $c\not=1$, $\log_c$ denotes the base $c$ logarithm, with $\log:=\log_{\e}$.

\section{On Theorem \ref{mainTheo}(2): Lower bounds on \texorpdfstring{$\enapp(G)$}{endapp(G)} and \texorpdfstring{$\affapp(G)$}{affapp(G)}}\label{sec2}

The following simple lemma, whose main message is that just as in the abelian case, all affine maps on finite groups are "difference-preserving", can be used in arguments for both upper and lower bounds on $\affapp(G)$:

\begin{lemma}\label{differenceLem}
Let $G$ be a group, $X$ a nonempty subset of $G$, and $f$ a function on $G$. The following are equivalent:

\begin{enumerate}
\item There exists an affine map $A$ on $G$ such that $f_{\mid X}=A_{\mid X}$.
\item There exists an endomorphism $\varphi$ of $G$ such that for all $x,y\in X$, we have $\varphi(y^{-1}x)=f(y)^{-1}f(x)$.
\item There exists an endomorphism $\varphi$ of $G$ and an element $x\in X$ such that for all $y\in X$, $\varphi(y^{-1}x)=f(y)^{-1}f(x)$.
\end{enumerate}
\end{lemma}

\begin{proof}
For \enquote{(1) $\Rightarrow$ (2)}: Write $A=\A_{g,\varphi}$. Then for all $x,y\in X$, it follows that
\[
f(y)^{-1}f(x)=A(y)^{-1}A(x)=(g\varphi(y))^{-1}g\varphi(x)=\varphi(y)^{-1}\varphi(x)=\varphi(y^{-1}x),
\]
as required.

For \enquote{(2) $\Rightarrow$ (3)}: This is clear.

For \enquote{(3) $\Rightarrow$ (1)}: Set $g:=f(x)\varphi(x)^{-1}$. Then for all $y\in X$, it follows that
\[
f(y)=f(x)\varphi(y^{-1}x)^{-1}=f(x)\varphi(x)^{-1}\varphi(y)=g\varphi(y),
\]
so that $A:=\A_{g,\varphi}$ does the job.
\end{proof}

The following is also useful for reduction arguments:

\begin{lemma}\label{reductionLem}
Let $G$ be a finite group, and $f$ a function on $G$. If $A$ is any \emph{bijective} affine map on $G$, then $\affapp(f)=\affapp(f\circ A)=\affapp(A\circ f)$.
\end{lemma}

\begin{proof}
Noting that $A^{-1}$ is a bijective affine map on $G$ as well, one sees that it suffices to show $\affapp(f)\leq\affapp(A\circ f)$ and $\affapp(f)\leq\affapp(f\circ A)$. To this end, let $X\subseteq G$ with $|X|=\affapp(f)$ and $B\in\Aff(G)$ with $f_{\mid X}=B_{\mid X}$. Then $(A\circ f)_{\mid X}=(A\circ B)_{\mid X}$, which proves the first inequality, and $(f\circ A)_{\mid A^{-1}[X]}=(B\circ A)_{\mid A^{-1}[X]}$, which proves the second inequality.
\end{proof}

From Lemma \ref{differenceLem}, one can immediately derive a sufficient criterion for the simultaneous validity of $\enapp(G)\geq\ell$ and $\affapp(G)\geq\ell+1$ for some fixed $\ell\in\IN^+$ based on the following concepts:

\begin{definition}\label{universalDef}
Let $G$ be a group.

\begin{enumerate}
\item A \emph{universal element in $G$} is an element $u\in G$ such that for all $g\in G$, there exists $\varphi=\varphi_g\in\End(G)$ such that $\varphi(u)=g$.
\item More generally, for $\ell\in\IN^+$, a \emph{universal $\ell$-tuple in $G$} is an $\ell$-tuple $(u_1,\ldots,u_{\ell})\in G^\ell$ such that for all $(g_1,\ldots,g_{\ell})\in G^{\ell}$, there exists $\varphi=\varphi_{(g_1,\ldots,g_{\ell})}\in\End(G)$ such that $\varphi(u_i)=g_i$ for $i=1,\ldots,\ell$.
\end{enumerate}
\end{definition}

Note that a universal element in a finite group $G$ is necessarily of order $\exp(G)$.

\begin{proposition}\label{universalProp}
Let $G$ be a nontrivial finite group.

\begin{enumerate}
\item If $G$ has a universal element, then $\enapp(G)\geq 1$ and $\affapp(G)\geq 2$.
\item More generally, if, for some $\ell\in\IN^+$, $G$ has a universal $\ell$-tuple, then $\enapp(G)\geq\ell$ and $\affapp(G)\geq\ell+1$.
\end{enumerate}
\end{proposition}

\begin{proof}
It suffices to show point (2). Let $(u_1,\ldots,u_\ell)$ be a universal $\ell$-tuple in $G$. Then $\enapp(G)\geq\ell$ holds because by the definition of \enquote{universal $\ell$-tuple}, every function on $G$ agrees with a suitable endomorphism of $G$ on the set $\{u_1,\ldots,u_\ell\}$ (and the $u_i$ are necessarily pairwise distinct). For the bound on $\affapp(G)$, it suffices to show that any function on $G$ agrees with some affine map on $G$ on the subset $\{1,u_1^{-1},\ldots,u_{\ell}^{-1}\}$. But for this, it is, by Lemma \ref{differenceLem}, sufficient to check that there is some endomorphism $\varphi$ of $G$ such that for $i=1,\ldots,\ell$, we have
\[
\varphi(u_i)=\varphi((u_i^{-1})^{-1}\cdot 1)=f(u_i^{-1})^{-1}f(1),
\]
which is clear by universality.
\end{proof}

We note two interesting consequences of Proposition \ref{universalProp}, the latter of which also directly implies Theorem \ref{mainTheo}(2):

\begin{corollary}\label{abelianCor}
Let $G$ be a nontrivial finite \emph{abelian} group. Then $\enapp(G)\geq 1$ and $\affapp(G)\geq 2$.
\end{corollary}

\begin{proof}
This follows from Proposition \ref{universalProp}(1), since by the structure theorem for finite abelian groups, it is clear that every such group $G$ has a universal element (in fact, by \cite[4.2.7, p.~102]{Rob96a}, any cyclic subgroup of $G$ generated by an element of order $\exp(G)$ always admits a direct complement in $G$, so that any such element is universal).
\end{proof}

\begin{corollary}\label{homogeneousCor}
For $m,r\in\IN^+$, $m\geq 2$, we have $\enapp((\IZ/m\IZ)^r)\geq r$ and $\affapp((\IZ/m\IZ)^r)\geq r+1$.
\end{corollary}

\begin{proof}
This follows from Proposition \ref{universalProp}(2) by observing that the \enquote{standard} generators of $(\IZ/m\IZ)^r$ form a universal $r$-tuple in that group.
\end{proof}

\begin{proof}[Proof of Theorem \ref{mainTheo}(2)]
This follows from Corollary \ref{homogeneousCor} with $m:=2$.
\end{proof}

The rest of this section serves either as direct preparation for determining the precise values of $\enapp(G)$ and $\affapp(G)$ for small $G$ at the end of the section (see Proposition \ref{smallProp}) or to raise some other interesting points. First, let us note the following simple fact, which shows that in many finite groups, the problem of determining the minimum endomorphic approximability is trivial:

\begin{proposition}\label{zeroProp}
Let $G$ be a finite group. The following are equivalent:

\begin{enumerate}
\item $\enapp(G)\geq 1$.
\item $G$ has a universal element.
\end{enumerate}
\end{proposition}

\begin{proof}
For \enquote{(1) $\Rightarrow$ (2)}: We show the contraposition. So assume that $G$ has no universal element. Then we can choose, for every element $g\in G$, an element $f(g)\in G$ which is not an image of $g$ under any endomorphism of $G$. The resulting function $f$ on $G$ clearly has endomorphic approximability $0$.

For \enquote{(2) $\Rightarrow$ (1)}: This implication is part of Proposition \ref{universalProp}(1).
\end{proof}

Hence $\enapp(G)=0$ whenever $G$ has no universal element, which holds true for example when $G$ is any nonabelian finite simple group. The problem of determining the minimum \emph{affine} approximability of a function on a finite group seems less trivial. For example, below, we will give another criterion on nontrivial finite groups $G$ which is sufficient for $\affapp(G)\geq 2$ (Proposition \ref{dominatingProp}) and which holds, for example, for $G=\Alt(4)$, which has no universal element (see also Example \ref{dominatingEx}(1) below). The new criterion is based on the following concepts:

\begin{definition}\label{dominatingLem}
Let $G$ be a nontrivial finite group.

\begin{enumerate}
\item We call the orbits of the natural action of $\Aut(G)$ on $G$ the \emph{automorphism orbits of $G$}, and $\{1_G\}$ the \emph{trivial automorphism orbit of $G$}.
\item A \emph{dominating automorphism orbit of $G$} is a nontrivial automorphism orbit $O$ of $G$ such that $|O|>\frac{1}{2}(|G|-1)$.
\item A function $f$ on $G$ with $f(1)=1$ is called \emph{automorphism orbit avoiding} (henceforth abbreviated by \emph{a.o.a.}) if and only if for all $g\in G\setminus\{1\}$, $g$ and $f(g)$ lie in different automorphism orbits of $G$.
\end{enumerate}
\end{definition}

\begin{proposition}\label{dominatingProp}
Consider the following conditions on nontrivial finite groups $G$:

\begin{enumerate}
\item $G$ has a dominating automorphism orbit.
\item $G$ has no a.o.a.~functions which are bijective (i.e., permutations on $G$).
\item $\affapp(G)\geq 2$.
\end{enumerate}

Conditions (1) and (2) are equivalent, and either of them implies condition (3).
\end{proposition}

Note that having a dominating automorphism orbit is \emph{not} necessary for a finite group $G$ to satisfy $\affapp(G)\geq2$ (consider, for example, $G=\IZ/6\IZ$, and see Proposition \ref{smallProp} below). The following combinatorial lemma will be used in the proof of Proposition \ref{dominatingProp}:

\begin{lemma}\label{partitionLem}
For a partition $\Pcal$ on a nonempty finite set $M$, we say that a function $f$ on $M$ is \emph{$\Pcal$-avoiding} if and only if for all $x\in M$, $x$ and $f(x)$ lie in different partition classes from $\Pcal$. Then the following are equivalent:

\begin{enumerate}
\item One of the elements of $\Pcal$ is of size larger than $\frac{1}{2}|M|$.
\item There is no $\Pcal$-avoiding permutation on $M$.
\end{enumerate}
\end{lemma}

\begin{proof}
For \enquote{(1) $\Rightarrow$ (2)}: Let $P$ denote the unique element of $\Pcal$ of size larger than $\frac{1}{2}|M|$. Assume, by contradiction, that there is a $\Pcal$-avoiding permutation $f$ on $M$. Then $f$ would have to map $P$ injectively into the smaller set $M\setminus P$, which is impossible.

For \enquote{(2) $\Rightarrow$ (1)}: We show the contraposition of this implication, i.e., that there \emph{is} a $\Pcal$-avoiding permutation on $M$ if all partition classes from $\Pcal$ have size at most $\frac{1}{2}|M|$, by induction on $|M|$. To this end, one first verifies directly that this holds for $|M|\leq 5$. Now assume that $|M|\geq 6$. Note that $\Pcal$ must consist of at least two nonempty partition classes, and choose distinct $P_1,P_2\in\Pcal$ such that $|P_1|\geq|P_2|\geq|P|$ for all $P\in\Pcal\setminus\{P_1,P_2\}$. Fix $p_1\in P_1$ and $p_2\in P_2$, and set $M':=M\setminus\{p_1,p_2\}$ and $\Pcal':=\{P\setminus\{p_1,p_2\}\mid P\in\Pcal\}$. Then $\Pcal'$ is a partition of $M'$, and it still has the property that none of its members has size larger than $\frac{1}{2}|M'|$. Indeed, an element of $\Pcal'$ is either obtained from $P_1$ or $P_2$ by deleting an element and thus has size at most $|P_1|-1\leq\frac{1}{2}|M|-1=\frac{1}{2}(|M|-2)=\frac{1}{2}|M'|$, or it is equal to an element of $\Pcal$ distinct from $P_1$ and $P_2$, whence it can only have size at most $\frac{1}{3}|M|$, which is less than or equal to $\frac{1}{2}|M|-1$ by the assumption $|M|\geq 6$. Hence by the induction hypothesis, there exists a $\Pcal'$-avoiding permutation $g$ on $M'$, and it is clear that $f:=g\cup\{(p_1,p_2),(p_2,p_1)\}$ (less formally: the permutation on $M$ obtained by adding the transposition of $p_1$ and $p_2$ to $g$) is a $\Pcal$-avoiding permutation on $M$.
\end{proof}

\begin{proof}[Proof of Proposition \ref{dominatingProp}]
The equivalence of (1) and (2) follows immediately from Lemma \ref{partitionLem} with $M:=G\setminus\{1\}$ and $\Pcal$ the collection of nontrivial automorphism orbits of $G$.

For \enquote{(2) $\Rightarrow$ (3)}: Let $f$ be any function on $G$. We need to show that $f$ agrees with a suitable affine map on $G$ on some subset of $G$ of size at least $2$. Since all constant functions on $G$ are affine, this is clear if $f$ is \emph{not} a permutation on $G$, so assume that $f:G\rightarrow G$ is bijective. Composing $f$ with a suitable translation on $G$, we may, by Lemma \ref{reductionLem}, also assume w.l.o.g.~that $f(1)=1$. But since $G$ has no a.o.a.~permutations by assumption, it follows that for some $g\in G\setminus\{1\}$ and some automorphism $\alpha$ of $G$, $f(g)=\alpha(g)$. Hence $f$ agrees with $\alpha$ on $\{1,g\}$, and we are done.
\end{proof}

\begin{example}\label{dominatingEx}
We now give some applications of Proposition \ref{dominatingProp}, some of which will also be used later.

\begin{enumerate}
\item The alternating group $\Alt(4)$ has no universal element, whence $\enapp(\Alt(4))=0$ by Proposition \ref{zeroProp}. On the other hand, $\Alt(4)$ has a dominating automorphism orbit, namely the one consisting of elements of order $3$, which has size $8$. Hence $\affapp(\Alt(4))\geq 2$ by Proposition \ref{dominatingProp}.
\item In any finite dihedral group $\D_{2n}=\langle r,s\mid r^n=s^2=1,srs^{-1}=r^{-1}\rangle$ with $n\geq 3$, the \emph{reflections}, i.e., the elements of the form $sr^k$ with $k\in\{0,\ldots,n-1\}$, form a dominating automorphism orbit, so $\affapp(\D_{2n})\geq 2$ for all $n\geq 3$. Similarly, one shows that $\affapp(\Dic_{4n})\geq 2$ for $n\geq 2$. Note that if $n$ is odd, neither $\D_{2n}$ nor $\Dic_{4n}$ have an element of order their respective exponent, so in particular, they have no universal elements.
\item Let $p$ be an odd prime. It is known that for every finite extraspecial group $G=p_+^{1+2n}$ of exponent $p$, the non-central elements of $G$ form a single automorphism orbit of $G$; this can be seen, for example, by using \cite[Theorem 1(a)]{Win72a} and the facts that $\Sp_{2n}(p)$ acts transitively on $\IF_p^{2n}\setminus\{0\}$ and that, since $G/G'\cong\IF_p^{2n}$ and $G'=\zeta G\cong\IF_p$, the group of central automorphisms of $G$ acts transitively on each nontrivial coset of $\zeta G$. Clearly, $G\setminus\zeta G$ is a dominating automorphism orbit of $G$ (and also, every element of it is a universal element of $G$), so $\affapp(G)=\affapp(p_+^{1+2n})\geq 2$.
\end{enumerate}
\end{example}

By what we know so far, the following can be deduced easily:

\begin{proposition}\label{pCubedProp}
Let $p$ be a prime and $G$ a finite group of order $p^k$, $k\in\{1,2,3\}$. Then $\affapp(G)\geq 2$.
\end{proposition}

\begin{proof}
This follows from the classification of $p$-groups up to order $p^3$, Corollary \ref{abelianCor}, Example \ref{dominatingEx} and the following argument for the extraspecial groups $G=p_-^{1+2}$ of order $p^3$ and exponent $p^2$ (for $p$ an odd prime): Write $G=\langle x,t\mid x^{p^2}=t^p=1,txt^{-1}=x^{1+p}\rangle$. It is not difficult to check (using the above presentation of $G$) that $\Aut(G)$ acts transitively on the set of generators of $\langle x\rangle\cong\IZ/p^2\IZ$ and that there is a (surjective) group homorphism $G\rightarrow\IZ/p\IZ$ under which $x$ \enquote{survives} (i.e., is mapped to a generator of $\IZ/p\IZ$). Hence, as all elements of $G\setminus\langle x\rangle$ have order $p$, $x$ is a universal element in $p_-^{1+2}$, so that $\affapp(p_-^{1+2})\geq 2$ by Proposition \ref{universalProp}(1).
\end{proof}

Note, however, that not all nontrivial finite $p$-groups have $\affapp$-value at least $2$, as the examples by means of which we will prove Theorem \ref{mainTheo}(3) are of order $p^8$.

We can also say something about $\enapp$- and $\affapp$-values of finite cyclic groups:

\begin{lemma}\label{cyclicLem}
Let $n\in\IN^+$. Then the following hold:

\begin{enumerate}
\item $\enapp(\IZ/n\IZ)=1$.
\item If $n$ is a prime, then $\affapp(\IZ/n\IZ)=2$.
\item If $n=p^k$, for a prime $p$ and some $k\in\IN^+$, then $\affapp(\IZ/n\IZ)=\affapp(\IZ/p^k\IZ)\leq p$.
\end{enumerate}
\end{lemma}

\begin{proof}
For (1): Note that by Corollary \ref{abelianCor}, $\enapp(\IZ/n\IZ)\geq 1$, so it suffices to give an example of a function $f$ on $\IZ/n\IZ$ such that $\enapp(f)\leq 1$. To this end, fix a generator $g$ of $\IZ/n\IZ$, and consider the following function $f$ on $\IZ/n\IZ$: It maps all non-generators of $\IZ/n\IZ$ (i.e., elements that generate a proper subgroup) to $g$, and it maps each generator $x$ of $\IZ/n\IZ$ to $x^2$ (squaring modulo $n$). Then any set $X\subseteq\IZ/n\IZ$ on which $f$ agrees with some endomorphism $\varphi$ of $\IZ/n\IZ$, say $\varphi(t)=a\cdot t$ for a suitable fixed $a\in\IZ/n\IZ$ and all $t\in\IZ/n\IZ$, cannot contain any non-generators, and it also cannot contain two distinct generators $x_1,x_2$, since that would imply $x_1^2=f(x_1)=\varphi(x_1)=ax_1$, and thus $a=x_1$, although one can analogously show $a=x_2$ as well.

For (2): Again by Corollary \ref{abelianCor}, we have $\affapp(\IZ/n\IZ)\geq 2$, so it suffices to give an example of a function $f$ on $\IZ/n\IZ$ with $\affapp(f)\leq 2$. Let $f$ be the square function of the ring $\IZ/n\IZ$. Note that any fixed affine map on $\IZ/n\IZ$ is of the form $x\mapsto ax+b$ with $a,b\in\IZ/n\IZ$ fixed, so that the elements of $\IZ/n\IZ$ on which $f$ and that affine map agree are the solutions to the quadratic equation $x^2=ax+b$ in the ring $\IZ/n\IZ$. Since $n$ is a prime, that ring is a field, whence each such equation has at most $2$ solutions in $\IZ/n\IZ$, as required.

For (3): We give an example of a function $f$ on $\IZ/p^k\IZ$ such that $\affapp(f)\leq p$. To define $f$, take as the underlying set of the group $\IZ/p^k\IZ$ the set of standard representatives $0,1,\ldots,p^k-1$ of the integer residue classes modulo $p^k$, and as the group operation addition modulo $p^k$. By integer division, every element $x$ of $\IZ/p^k\IZ$ can be uniquely written as $\rem(x)+\quot(x)\cdot p$ with $\rem(x)\in\{0,1,\ldots,p-1\}$ and $\quot(x)\in\{0,1\ldots,p^{k-1}-1\}$. Then we define $f$ through $f(x):=\rem(x)+\quot(x)$. Let us argue why $f$ cannot agree with an affine map of $\IZ/p^k\IZ$ on any subset $X$ with $|X|\geq p+1$. Indeed, such a subset $X$ contains two distinct elements $x$ and $y$ such that $\rem(x)=\rem(y)$, say w.l.o.g.~$x\geq y$ in $\IZ$. Then by Lemma \ref{differenceLem}, it would follow that some endomorphism $\varphi$ of $\IZ/p^k\IZ$ maps $x-y=p\cdot(\quot(x)-\quot(y))$ to $f(x)-f(y)=\quot(x)-\quot(y)$, which is impossible, because the order in $\IZ/p^k\IZ$ of $\quot(x)-\quot(y)$ is strictly larger than the order of $p\cdot(\quot(x)-\quot(y))$.
\end{proof}

We are now ready for determining the precise $\enapp$- and $\affapp$-values of groups of order up to $7$:

\begin{proposition}\label{smallProp}
The precise values of $\enapp(G)$ and $\affapp(G)$ for finite groups $G$ with $|G|\leq 7$ are as in Table \ref{table1}.

\begin{table}[H]
\caption{List of $\enapp(G)$ and $\affapp(G)$ for $|G|\leq7$.}
\label{table1}
\begin{center}
\begin{tabular}{|c|c|c|c|c|c|c|c|c|c|}\hline
$G$ & $\{1\}$ & $\IZ/2\IZ$ & $\IZ/3\IZ$ & $\IZ/4\IZ$ & $(\IZ/2\IZ)^2$ & $\IZ/5\IZ$ & $\IZ/6\IZ$ & $\Sym(3)$ & $\IZ/7\IZ$ \\ \hline
$\enapp(G)$ & $1$ & $1$ & $1$ & $1$ & $2$ & $1$ & $1$ & $0$ & $1$ \\ \hline
$\affapp(G)$ & $1$ & $2$ & $2$ & $2$ & $3$ & $2$ & $2$ & $2$ & $2$ \\ \hline
\end{tabular}
\end{center}
\end{table}
\end{proposition}

\begin{proof}[Proof of Proposition \ref{smallProp}]
By Lemma \ref{cyclicLem}, all the assertions on $\enapp(G)$ and $\affapp(G)$ in the cases where $G$ is cyclic are clear except for the one assertion $\affapp(\IZ/6\IZ)=2$. To see that this holds, it suffices to give an example of a function $f$ on $G=\IZ/6\IZ$ such that $\affapp(f)=2$. Using that $\IZ/6\IZ\cong\IZ/2\IZ\times\IZ/3\IZ$, we write the elements of $G$ as $(x,y)$ with $x\in\{0,1\}$ and $y\in\{0,1,2\}$. Consider the following function $f$ on $G$, which is a kind of \enquote{component swap}: $f(x,y):=(y\mod{2},x)$ for all $x\in\{0,1\}$ and $y\in\{0,1,2\}$. We argue why $f$ cannot agree with any affine function of $G$ on any subset of size $3$.

Note that since the component subgroups $\IZ/2\IZ$ and $\IZ/3\IZ$ of $G$ are fully invariant, any affine map of $G$ can be written as a product (in the sense of component-wise application) of an affine map on $\IZ/2\IZ$ and an affine map on $\IZ/3\IZ$. Consequently, any affine map of $G$ maps pairs with the same first (resp.~second) coordinate to pairs whose first (resp.~second) coordinates agree as well. But by its definition, $f$ never maps distinct pairs with the same second coordinate (which are necessarily of the form $(0,x)$ and $(1,x)$ for some $x\in\{0,1,2\}$) to such pairs. Hence if there are any three pairwise distinct elements $(x_1,y_1),(x_2,y_2),(x_3,y_3)$ of $G$ on which $f$ agrees with some affine map, then their second coordinates must be pairwise distinct, so that we can assume w.l.o.g.~that $y_1=0$, $y_2=1$ and $y_3=2$.

So it remains to show that no affine map $A=\A_{g,\varphi}$ on $G$ can show the following mapping behavior for some $x_1,x_2,x_3\in\{0,1\}$: $(x_1,0)\mapsto (0,x_1),(x_2,1)\mapsto(1,x_2)$ and $(x_3,2)\mapsto(0,x_3)$. If $\{x_1,x_2,x_3\}=\{0,1\}$, then choosing $p_1,p_2\in\{(x_1,0),(x_2,1),(x_3,2)\}$ with the same first coordinate and using that by the proof of Lemma \ref{differenceLem}, $\varphi(p_1-p_2)=f(p_1)-f(p_2)$, we see that $\varphi_{\mid\IZ/3\IZ}$ is trivial, and thus $A_{\mid\IZ/3\IZ}$ is constant, contradicting the fact that in the images of the three pairs $(x_1,0),(x_2,1)$ and $(x_3,2)$, there appear two distinct values in the second coordinates. Hence $x_1=x_2=x_3$ so that the three images would need to have the same first coordinate, which is not the case, the final contradiction.

We now turn to the two non-cyclic groups in the list: $(\IZ/2\IZ)^2$ and $\Sym(3)\cong\D_6$.

For $G=(\IZ/2\IZ)^2$, write the elements of $G$ as $(x,y)$ with $x,y\in\{0,1\}$, and note that by Corollary \ref{homogeneousCor}, $\enapp(G)\geq 2$ and $\affapp(G)\geq 3$. It is easy to check that $\enapp(f)=2$ for the following function $f$ on $(\IZ/2\IZ)^2$:

\begin{table}[H]
\caption{A function $f$ on $(\IZ/2\IZ)^2$ with $\enapp(f)=2$.}
\label{table2}
\begin{center}
\begin{tabular}{|c|c|c|c|c|}\hline
$x$ & $(0,0)$ & $(1,0)$ & $(0,1)$ & $(1,1)$ \\ \hline
$f(x)$ & $(1,0)$ & $(1,0)$ & $(0,1)$ & $(0,1)$ \\ \hline
\end{tabular}
\end{center}
\end{table}

For $\affapp(G)=3$, we only need to argue that there is some non-affine function on $G$, which is clear, since $G$ has precisely $2^{2^2}=16$ endomorphisms, thus precisely $4\cdot 16=64$ affine maps, but there are $4^4=256$ functions on $G$ altogether.

Finally, for $G=\Sym(3)\cong\D_6$, we note that $\enapp(G)=0$ holds by Proposition \ref{zeroProp}, since $G$ has no elements of order $\exp(G)=6$ and thus no universal elements. Moreover, $\affapp(G)\geq 2$ by Example \ref{dominatingEx}(2), so it suffices to give an example of a function $f$ on $G$ such that $\affapp(f)=2$. Since $G=\D_6=\langle r,s\mid r^3=s^2=1,srs^{-1}=r^{-1}\rangle$, we can write the elements of $G$ in normal form as $1,r,r^2,s,sr,sr^2$. We define $f$ via the following table:

\begin{table}[H]
\caption{A function $f$ on $\D_6$ with $\affapp(f)=2$.}
\label{table3}
\begin{center}
\begin{tabular}{|c|c|c|c|c|c|c|}\hline
$x$ & $1$ & $r$ & $r^2$ & $s$ & $sr$ & $sr^2$ \\ \hline
$f(x)$ & $1$ & $r$ & $r$ & $1$ & $r^2$ & $r^2$ \\ \hline
\end{tabular}
\end{center}
\end{table}

Let $X\subseteq G$ with $|X|=3$. We show by contradiction that $f$ cannot agree with any affine map $A=\A_{g,\varphi}$ of $G$ on $X$. First, consider the case when $X$ consists only of rotations, i.e., $X=\{1,r,r^2\}$. Then if $f_{\mid X}=A_{\mid X}$, we would in particular have $A(1)=1$ and thus that $A=\varphi$ is an endomorphism of $G$, but then the mapping behavior of $A$ on $\{r,r^2\}$ is contradictory. Next, consider the case when $X$ contains both a rotation $r^k$ and a reflection $sr^l$ for suitable $k,l\in\{0,1,2\}$. Then by the proof of Lemma \ref{differenceLem}, $\varphi(sr^{k+l})=\varphi(r^{-k}sr^l)=f(r^k)^{-1}f(sr^l)\in\langle r\rangle$, which is only possible if $\varphi$ is the trivial endomorphism of $G$ so that $A$ is constant. But by the definition of $f$, $A$ assumes at least two distinct values on the three elements of $X$, another contradiction. The only case left is when $X$ only consists of reflections, i.e., $X=\{s,sr,sr^2\}$. Denoting by $\mu_s$ the left translation by $s$ on $G$, we get in this case that $A\circ\mu_s$ is an affine map of $G$ showing the following mapping behavior: $1\mapsto 1,r\mapsto r^2,r^2\mapsto r^2$. From this, we can derive a contradiction like we did in the case $X=\{1,r,r^2\}$.
\end{proof}

\section{On Theorem \ref{mainTheo}(1): Upper bounds on \texorpdfstring{$\enapp(G)$}{endapp(G)} and \texorpdfstring{$\affapp(G)$}{affapp(G)}}\label{sec3}

Recall the general approximability notion $\app_{\F}(M_1,M_2)$ from Definition \ref{approxDef}. We will now show the following general combinatorial lemma:

\begin{lemma}\label{genAppLem}
Let $M_1$ and $M_2$ be finite sets of cardinality at least $2$, $f:\IN\rightarrow\left(0,\infty\right)$ and $\F$ a set of functions $M_1\rightarrow M_2$ such that $|\F|\leq|M_2|^{f(|M_1|)}$. Then

\[
\app_{\F}(M_1,M_2)\leq\max\{\e^2\cdot\frac{|M_1|}{|M_2|},f(|M_1|)\log{|M_2|}+\log{|M_1|}\}.
\]
\end{lemma}

\begin{proof}
Consider the Hamming metric $\dist$ on the set $M_2^{M_1}$ of all functions $M_1\rightarrow M_2$, defined as follows:
\[
\dist(f,g):=|\{x\in M_1\mid f(x)\not=g(x)\}|.
\]
For $k\in\IN$ and $\G\subseteq M_2^{M_1}$, denote by
\[
\N_k(\G):=\{h\in M_2^{M_1}\mid \exists g\in\G: \dist(g,h)\leq k\}
\]
the \emph{$k$-neighborhood of $\G$ in $M_2^{M_1}$ with respect to $\dist$}. For $g\in M_2^{M_1}$, we also write $\N_k(g)$ instead of $\N_k(\{g\})$ for the \emph{$k$-ball around $g$}, and we denote by
\[
\C_k(g):=\{h\in M_2^{M_1}\mid \dist(g,h)=k\}
\]
the \emph{$k$-circle around $g$}. $|\N_k(g)|$ and $|\C_k(g)|$ do not depend on $g$; indeed,
\[
|\C_k(g)|={|M_1| \choose k}\cdot (|M_2|-1)^k,
\]
and
\[
|\N_k(g)|=\sum_{i=0}^k{|\C_i(g)|}=\sum_{i=0}^k{{|M_1| \choose i}\cdot (|M_2|-1)^i}.
\]
Henceforth, we will denote by $\nu_k$ resp.~$\gamma_k$ the size of the $k$-ball resp.~$k$-circle around any given function $M_1\rightarrow M_2$.

The proof is based on the following observation: For each $k\in\{0,\ldots,|M_1|-1\}$, $\app_{\F}(M_1,M_2)\leq|M_1|-(k+1)$ is equivalent to the inclusion $\N_k(\F)\subseteq M_2^{M_1}$ being proper. We thus want to show $|\N_k(\F)|<|M_2^{M_1}|=|M_2|^{|M_1|}$ for $k$ as large as possible. Now
\[
|\N_k(\F)|=|\bigcup_{g\in\F}{\N_k(g)}|\leq|\F|\cdot\nu_k\leq|M_2|^{f(|M_1|)}\cdot\sum_{i=0}^k{\gamma_i}\leq|M_2|^{f(|M_1|)}\cdot|M_1|\cdot\max_{i=0,\ldots,k}{\gamma_i},
\]
and so, setting $L:=\log_{|M_2|}(|M_1|)$, for $|\N_k(\F)|<|M_2|^{|M_1|}$ to hold, it is sufficient to have
\[
\gamma_i={|M_1| \choose i}\cdot (|M_2|-1)^i<|M_2|^{|M_1|-f(|M_1|)-L}
\]
for $i=0,\ldots,k$. We make the ansatz $i=|M_1|-\ell$ and transform the substituted inequality
\begin{equation}\label{eq1}
{|M_1| \choose |M_1|-\ell}\cdot (|M_2|-1)^{|M_1|-\ell}<|M_2|^{|M_1|-f(|M_1|)-L}
\end{equation}
to obtain a (preferably small) lower bound on $\ell$, which is also an upper bound on $\app_{\F}(M_1,M_2)$. Using that
\[
{|M_1| \choose |M_1|-\ell}={|M_1| \choose \ell}\leq\frac{|M_1|^{\ell}}{\ell!}=\frac{|M_2|^{L\cdot\ell}}{\ell!},
\]
we see that for Formula (\ref{eq1}) to hold, it is sufficient to have
\begin{equation}\label{eq2}
|M_2|^{f(|M_1|)+\ell(L-1)+L}<\ell!.
\end{equation}
Now $\ell!>(\ell/\e)^{\ell}$ (see, for instance, \cite{Tao10a}), so Formula (\ref{eq2}) is implied by
\[
|M_2|^{f(|M_1|)+\ell(L-1)+L}\leq(\frac{\ell}{\e})^{\ell},
\]
which by taking logarithms on both sides and bringing all summands involving $\ell$ as a factor on one side is equivalent to
\begin{equation}\label{eq3}
(f(|M_1|)+L)\cdot\log{|M_2|}\leq \ell\cdot(\log{\ell}-1-(L-1)\log{|M_2|}).
\end{equation}
Finally, we note that for Formula (\ref{eq3}) to hold, it is sufficient to have both
\[
\ell\geq(f(|M_1|)+L)\cdot\log{|M_2|}
\]
and
\[
\log{\ell}-1-(L-1)\log{|M_2|}\geq 1,
\]
i.e.,
\[
\ell\geq\max\{(f(|M_1|)+L)\cdot\log{|M_2|},\e^2\cdot|M_2|^{L-1}\}=\max\{(f(|M_1|)+L)\cdot\log{|M_2|},\e^2\cdot\frac{|M_1|}{|M_2|}\}.
\]
This concludes the proof.
\end{proof}

\begin{proof}[Proof of Theorem \ref{mainTheo}(1)]
For $|G|=2,\ldots,7$, the validity of the asserted inequalities can be checked case by case using Proposition \ref{smallProp}, so we may assume that $|G|\geq 8$. Note that since endomorphisms of $G$ are determined by their values on any generating subset of $G$ and the size of a minimal (with respect to inclusion) generating subset of $G$ is always bounded from above by $\log_2{|G|}$ due to Lagrange's theorem, we have $|\End(G)|\leq|G|^{\log_2{|G|}}$ and $|\Aff(G)|=|G|\cdot|\End(G)|\leq|G|^{1+\log_2{|G|}}$.

We can therefore apply Lemma \ref{genAppLem}(2) with $M_1:=M_2:=G$ and $f:x\mapsto\log_2{x}$ (resp.~$f:x\mapsto 1+\log_2{x}$) to conclude that
\[
\enapp(G)\leq\max\{\e^2,\log_2{|G|}\log{|G|}+\log{|G|}\}=\max\{\e^2,(\frac{1}{\log{2}}+\frac{1}{\log{|G|}})\log^2{|G|}\}
\]
and
\[
\affapp(G)\leq\max\{\e^2,(1+\log_2{|G|)\log{|G|}+\log{|G|}}\}=\max\{\e^2,(\frac{1}{\log{2}}+\frac{2}{\log{|G|}})\log^2{|G|}\}.
\]
As it is easily checked that for all real numbers $x\geq 8$, one has
\[
(\frac{1}{\log{2}}+\frac{1}{x})\log^2{x}\geq\e^2,
\]
the asserted upper bounds on $\enapp(G)$ and $\affapp(G)$ follow from the just derived inequalities.
\end{proof}

\section{On Theorem \ref{mainTheo}(3): Finite groups \texorpdfstring{$G$}{G} minimizing both \texorpdfstring{$\enapp(G)$}{endapp(G)} and \texorpdfstring{$\affapp(G)$}{affapp(G)}}\label{sec4}

The following finite groups are the ones studied by Jonah and Konvisser in \cite{JK75a}, as mentioned in the Introduction:

\begin{definition}\label{jkDef}
Let $p$ be a prime, $\lambda=(\lambda_1,\lambda_2)$ an element of the set $\{(1,0),(0,1),\linebreak[4](1,1),\ldots,(p-1,1)\}$ of representatives of the $1$-dimensional subspaces of $\IF_p^2$. The \emph{JK-group} $\J_{p,\lambda}$ is defined as the $p$-group of nilpotency class $2$ generated by $4$ elements $a_1,a_2,b_1,b_2$ subject to the following additional relations:
\[
a_1^p=[a_1,b_1], a_2^p=[a_1,b_1^{\lambda_1}b_2^{\lambda_2}], b_1^p=[a_2,b_1b_2], b_2^p=[a_2,b_2], [a_1,a_2]=[b_1,b_2]=1.
\]
\end{definition}

We now collect some basic facts on the $\J_{p,\lambda}$, which Jonah and Konvisser already used in their proof that $\Aut(\J_{p,\lambda})$ is abelian. It is elementary to check that every element of $\J_{p,\lambda}$ has a unique normal form representation as

\[
a_1^{k_1}a_2^{k_2}b_1^{\ell_1}b_2^{\ell_2}[a_1,b_1]^{r_1}[a_1,b_2]^{r_2}[a_2,b_1]^{r_3}[a_2,b_2]^{r_4}
\]

with $k_1,k_2,\ell_1,\ell_2,r_1,\ldots,r_4\in\{0,\ldots,p-1\}$, so that $\J_{p,\lambda}$ is a special $p$-group of order $p^8$ with $\zeta\J_{p,\lambda}=\J_{p,\lambda}'$ elementary abelian of order $p^4$ with $\IF_p$-basis $[a_1,b_1],[a_1,b_2],[a_2,b_1],[a_2,b_2]$. The central quotient of $\J_{p,\lambda}$ is elementary abelian of order $p^4$ too, with $\IF_p$-basis the images of $a_1,a_2,b_1,b_2$ under the canonical projection.

Jonah and Konvisser showed that all automorphisms of $\J_{p,\lambda}$ are central, i.e., of the form $g\mapsto (\id+\varphi)(g):=g\varphi(g)$ for some fixed homomorphism $\varphi:\J_{p,\lambda}\rightarrow\zeta\J_{p,\lambda}$ (and conversely, all such maps on $\J_{p,\lambda}$ actually are automorphisms, so that the automorphisms of $\J_{p,\lambda}$ are completely understood), which immediately implies that any two automorphisms of $\J_{p,\lambda}$ commute, since $\im(\varphi)\leq\zeta\J_{p,\lambda}=\J_{p,\lambda}'\leq\ker(\varphi)$ for every homomorphism $\varphi:\J_{p,\lambda}\rightarrow\zeta\J_{p,\lambda}$, whence the composition of any two such homomorphisms is the trivial endomorphism of $\J_{p,\lambda}$.

The author's approach in \cite{Bor14a} to show that for \enquote{most} of the $\J_{p,\lambda}$, even any two \emph{endo}morphisms of $\J_{p,\lambda}$ commute, is analogous to the one of Jonah and Konvisser, i.e., it essentially consists of gaining a complete understanding of all endomorphisms of those $\J_{p,\lambda}$ in the form of the following key lemma:

\begin{lemma}\label{keyLem}
Let $p$ be an odd prime, $\lambda_1\in\{0,\ldots,p-1\}$, $\lambda_2:=1$ and $\lambda:=(\lambda_1,\lambda_2)$. Then every endomorphism of $\J_{p,\lambda}$ which is not an automorphism is a homomorphism $\J_{p,\lambda}\rightarrow\zeta\J_{p,\lambda}$. 
\end{lemma}

In other words, every endomorphism of such a $\J_{p,\lambda}$ is of one of the two forms $\varphi$ or $\id+\varphi$ for a homomorphism $\varphi:\J_{p,\lambda}\rightarrow\zeta\J_{p,\lambda}$, which implies the commutativity of $\End(\J_{p,\lambda})$ in a simple case distinction. It is also this complete understanding of the endomorphisms of \enquote{most} $\J_{p,\lambda}$ which will allow us to prove that both $\enapp(\J_{p,\lambda})=0$ and $\affapp(\J_{p,\lambda})=1$ (the latter via Lemma \ref{differenceLem}).

For the reader's convenience (and since it never appeared in a peer-reviewed publication before), we now recall the proof of Lemma \ref{keyLem} as in \cite{Bor14a}.

\begin{proof}[Proof of Lemma \ref{keyLem}]
Since $\J_{p,\lambda}$ is nilpotent of class $2$, $p$ is odd and $\J_{p,\lambda}'$ has exponent $p$, it follows that $\J_{p,\lambda}$ satisfies the identity $(xy)^p=x^py^p$ (see also \cite[5.3.5, p.~141]{Rob96a}). Using this, the defining relations and that $\lambda_2=1$ by assumption, it follows that
\begin{align}\label{powerEq}
&(a_1^{k_1}a_2^{k_2}b_1^{\ell_1}b_2^{\ell_2}[a_1,b_1]^{r_1}[a_1,b_2]^{r_2}[a_2,b_1]^{r_3}[a_2,b_2]^{r_4})^p= \notag \\
&[a_1,b_1]^{k_1+\lambda_1\cdot k_2}[a_1,b_2]^{k_2}[a_2,b_1]^{\ell_1}[a_2,b_2]^{\ell_1+l_2}.
\end{align}
But the (linear) map
\[
(\mathbb{Z}/p\mathbb{Z})^4\rightarrow(\mathbb{Z}/p\mathbb{Z})^4,(k_1,k_2,\ell_1,\ell_2)\mapsto (k_1+\lambda_1k_2,k_2,\ell_1,\ell_1+\ell_2),
\]
is a bijection. Hence by Formula (\ref{powerEq}), it follows that the elements of order a divisor of $p$ in $\J_{p,\lambda}$ are just those that lie in $\J_{p,\lambda}'=\zeta\J_{p,\lambda}$ and that each such element has precisely one $p$-th root in $\J_{p,\lambda}$ of the form $a_1^{k_1}a_2^{k_2}b_1^{\ell_1}b_2^{\ell_2}$.

Now let $\varphi$ be an endomorphism of $\J_{p,\lambda}$ with nontrivial kernel. Fix an element $x\in\ker(\varphi)$ of order $p$, and let $y=a_1^{s_1}a_2^{s_2}b_1^{t_1}b_2^{t_2}$, with $(s_1,s_2,t_1,t_2)\in\{0,\ldots,p-1\}^4\setminus\{(0,0,0,0)\}$, be a $p$-th root of $x$ in $\J_{p,\lambda}$. Then $\varphi$ maps $y$ to an element of $\J_{p,\lambda}$ of order a divisor of $p$, i.e., to an element of $\zeta\J_{p,\lambda}$.

Note that for showing $\im(\varphi)\subseteq\zeta\J_{p,\lambda}$, by the defining relations of $\J_{p,\lambda}$, it suffices to show that at least one of the three elements $a_1,a_2,b_1$ gets mapped into $\zeta\J_{p,\lambda}$ by $\varphi$. Moreover, for any element $g\in\J_{\lambda,p}$, if $\varphi(g)\in\zeta\J_{\lambda,p}$, then all commutators of the form $[h,g]$ with $h\in\J_{\lambda,p}$ are in $\ker(\varphi)$.

We now agree on the following notational conventions: \enquote{$(k_1,k_2,\ell_1,\ell_2)\mapsto\zeta$} is an abbreviation for \enquote{$\varphi(a_1^{k_1}a_2^{k_2}b_1^{\ell_1}b_2^{\ell_2})\in\zeta\J_{\lambda,p}$}, and \enquote{$(K_1,K_2,L_1,L_2)\mapsto 1$} abbreviates \enquote{$[a_1,b_1]^{K_1}[a_1,b_2]^{K_2}[a_2,b_1]^{L_1}[a_2,b_2]^{L_2}\in\ker(\varphi)$}. Then the following implications hold: Firstly, by \enquote{taking brackets} with the generators $a_1,a_2,b_1,b_2$,
\begin{align}\label{firstEq}
(k_1,k_2,\ell_1,\ell_2)\mapsto\zeta \Rightarrow &(\ell_1,\ell_2,0,0)\mapsto 1,(0,0,\ell_1,\ell_2)\mapsto 1,(k_1,0,k_2,0)\mapsto 1\hspace{3pt}\text{and} \notag \\
&(0,k_1,0,k_2)\mapsto 1.
\end{align}
Secondly, using Formula (\ref{powerEq}), the observation that an element is mapped into the center if and only if its $p$-th power is in the kernel of $\varphi$ translates to

\[
(k_1,k_2,\ell_1,\ell_2)\mapsto\zeta \Leftrightarrow (k_1+\lambda_1k_2,k_2,\ell_1,\ell_1+\ell_2)\mapsto 1,
\]

which is equivalent to

\begin{equation}\label{secondEq}
(K_1,K_2,L_1,L_2)\mapsto 1 \Leftrightarrow (K_1-\lambda_1K_2,K_2,L_1,L_2-L_1)\mapsto\zeta.
\end{equation}

Our assumption that $\varphi(y)\in\zeta\J_{p,\lambda}$ translates to $(s_1,s_2,t_1,t_2)\mapsto\zeta$. We now make a case distinction according to the values of $s_1,s_2,t_1,t_2$. First, assume that $t_1\not=0$. By Formula (\ref{firstEq}), we have $(0,0,t_1,t_2)\mapsto 1$, and by Formula (\ref{secondEq}), this implies $(0,0,t_1,t_2-t_1)\mapsto\zeta$. Applying Formula (\ref{firstEq}), we deduce from this that $(0,0,t_1,t_2-t_1)\mapsto 1$. Iteration of this argumentation yields $(0,0,t_1,t_2-n\cdot t_1)\mapsto 1$ for all $n\in\IN$, i.e., $(0,0,t_1,t)\mapsto 1$ for all $t\in\{0,\ldots,p-1\}$. In particular, $(0,0,t_1,0)\mapsto 1$, which by $t_1\not=0$ implies $(0,0,1,0)\mapsto 1$, and thus $(0,0,1,-1)\mapsto\zeta$ by Formula (\ref{secondEq}). Spelled out, this means $\varphi(b_1b_2^{-1})\in\zeta\J_{\lambda,p}$. But also, by an appropriate subtraction among the $(0,0,t_1,t_1-n\cdot t_2)\mapsto 1$, we derive $(0,0,0,1)\mapsto 1$, which by Formula (\ref{secondEq}) yields $(0,0,0,1)\mapsto\zeta$, or explicitly, $\varphi(b_2)\in\zeta\J_{\lambda,p}$. In combination with the already derived $\varphi(b_1b_2^{-1})\in\zeta\J_{\lambda,p}$, this yields $\varphi(b_1)\in\zeta\J_{\lambda,p}$, so that we are done in this case.

Now assume $t_1=0$. The subcase $s_2\not=0$ reduces to the first case, since Formula (\ref{firstEq}) yields $(s_1,0,s_2,0)\mapsto 1$, which in turn gives $(s_1,0,s_2,-s_2)\mapsto\zeta$ by Formula (\ref{secondEq}). Hence we can assume that $t_1=s_2=0$. But then another successive application of Formulas (\ref{firstEq}) and (\ref{secondEq}) yields $(s_1,0,0,0)\mapsto\zeta$, which in case $s_1\not=0$ implies $(1,0,0,0)\mapsto\zeta$, i.e., the sufficient $\varphi(a_1)\in\zeta\J_{\lambda,p}$. Hence we can assume $s_1=s_2=t_1=0$ and $t_2\not=0$. In this case, $(t_1,t_2,0,0)\mapsto 1$, valid by Formula (\ref{firstEq}), simplifies to $(0,t_2,0,0)\mapsto 1$, which by Formula (\ref{secondEq}) yields $(-\lambda_1t_2,t_2,0,0)\mapsto\zeta$ and hence reduces the situation to the case $t_1=0,s_2\not=0$ dealt with before.
\end{proof}

By Proposition \ref{zeroProp}, it is now clear that the $\J_{p,\lambda}$ with $p>2$ and $\lambda\not=(1,0)$ satisfy $\enapp(\J_{p,\lambda})=0$. Indeed, each $x\in\J_{p,\lambda}$ can only be mapped to elements from $\zeta\J_{p,\lambda}\cup x\zeta\J_{p,\lambda}\subsetneq\J_{p,\lambda}$ under endomorphisms of $\J_{p,\lambda}$ by Lemma \ref{keyLem} and Jonah and Konvisser's results, so that no such $x$ is universal in $\J_{p,\lambda}$. It remains to show $\affapp(\J_{p,\lambda})=1$, which is done by proving the following proposition:

\begin{proposition}\label{finalProp}
Let $p$ be an odd prime, $\lambda_1\in\{0,\ldots,p-1\}$, $\lambda_2:=1$, $\lambda:=(\lambda_1,\lambda_2)$, and $\sigma$ any fixed-point free automorphism of $(\IZ/p\IZ)^4=\IF_p^4$ (for example, $\sigma$ could be chosen as a Singer cycle). Denote, for $i=1,2,3,4$, by $\pi_i:(\IZ/p\IZ)^4\rightarrow\IZ/p\IZ$ the projection onto the $i$-th coordinate. Then the following function $f$ on $\J_{p,\lambda}$, defined on elements in normal form, satsfies $\affapp(f)=1$:

\begin{align*}
&f(a_1^{k_1}a_2^{k_2}b_1^{\ell_1}b_2^{\ell_2}[a_1,b_1]^{r_1}[a_1,b_2]^{r_2}[a_2,b_1]^{r_3}[a_2,b_2]^{r_4}):= \\
&a_1^{\pi_1(\sigma(k_1,k_2,\ell_1,\ell_1))}a_2^{\pi_2(\sigma(k_1,k_2,\ell_1,\ell_2))}b_1^{\pi_3(\sigma(k_1,k_2,\ell_1,\ell_2))}b_2^{\pi_4(\sigma(k_1,k_2,\ell_1,\ell_2))}\cdot \\
&[a_1,b_1]^{\pi_1(\sigma(r_1,r_2,r_3,r_4))}[a_1,b_2]^{\pi_2(\sigma(r_1,r_2,r_3,r_4))}[a_2,b_1]^{\pi_3(\sigma(r_1,r_2,r_3,r_4))}[a_2,b_2]^{\pi_4(\sigma(r_1,r_2,r_3,r_4))}.
\end{align*}
\end{proposition}

In other words, if we identify the elements of $\J_{p,\lambda}$ via their normal forms with octuples $(k_1,k_2,\ell_1,\ell_2,r_1,r_2,r_3,r_4)\in\IF_p^8$, then $f$ consists of applying $\sigma$ to both $(k_1,k_2,\ell_1,\ell_2)$ and $(r_1,r_2,r_3,r_4)$ and concatenating the resulting images.

\begin{proof}[Proof of Proposition \ref{finalProp}]
By Lemma \ref{differenceLem}, we need to show that there do not exist two distinct elements $x,y\in\J_{p,\lambda}$ such that some endomorphism of $\J_{p,\lambda}$ maps $y^{-1}x$ to $f(y)^{-1}f(x)$. We do so in a case distinction:

\begin{enumerate}
\item Case: $x,y$ lie in a common coset of $\zeta\J_{p,\lambda}$. Then by definition of $f$ and choice of $\sigma$, $y^{-1}x$ and $f(y)^{-1}f(x)$ are distinct nontrivial elements of $\zeta\J_{p,\lambda}$. By Jonah and Konvisser's results, any automorphism of $\J_{p,\lambda}$ is of the form $\id+\varphi$ with $\zeta\J_{p,\lambda}\leq\ker(\varphi)$, in particular fixes $\zeta\J_{p,\lambda}$ element-wise, and by Lemma \ref{keyLem}, any endomorphism of $\J_{p,\lambda}$ which is not an automorphism maps all elements of $\zeta\J_{p,\lambda}$ to $1$. Hence no endomorphism of $\J_{p,\lambda}$ can map $y^{-1}x$ to $f(y)^{-1}f(x)$ in this case, as required.
\item Case: $x,y$ lie in different cosets of $\zeta\J_{p,\lambda}$. Then by definition of $f$ and choice of $\sigma$, $y^{-1}x$ and $f(y)^{-1}f(x)$ lie in different cosets of $\zeta\J_{p,\lambda}$, both distinct from $\zeta\J_{p,\lambda}$ itself. Jonah and Konvisser's result that all automorphisms of $\J_{p,\lambda}$ are central just means that they leave all cosets of $\zeta\J_{p,\lambda}$ invariant, and by Lemma \ref{keyLem}, all other endomorphisms have their image contained in $\zeta\J_{p,\lambda}$, so no endomorphism mapping $y^{-1}x$ to $f(y)^{-1}f(x)$ can exist in this case either.
\end{enumerate}
\end{proof}

\section{Computation of \texorpdfstring{$\enapp(G)$}{endapp(G)} and \texorpdfstring{$\affapp(G)$}{affapp(G)} for small \texorpdfstring{$G$}{G}}\label{sec5}

This section can be seen as an extension of Proposition \ref{smallProp}. While Proposition \ref{smallProp}, just as everything so far, was based entirely on theoretical arguments, the results of this section are computational and were obtained using implementations of certain algorithms in GAP \cite{GAP4}.

Those algorithms are simple backtracking searches, which can be explained in a broader context. Let $X$ be a finite set of size $m$, and fix a repetition-free list $x_1,x_2,\ldots,x_m$ of the elements of $X$. Moreover, let $\F$ be a family of functions on $X$. In order to compute $\app_{\F}(X)$, we work with certain encodings $\tilde{f}$ of functions $f$ which map from a subset of $X$ of the form $\{x_1,x_2,\ldots,x_{\ell}\}$, for some $\ell\in\{1,\ldots,m\}$, to $X$. These encodings are defined as follows: $\tilde{f}$ is an $\ell$-tuple with entries in $\{1,\ldots,m\}$, and for all $i=1,\ldots,\ell$, if $j$ is the $i$-th entry of $\tilde{f}$, then $f(x_i)=x_j$.

For each ordered tuple $\tilde{f}$ with entries in $\{1,\ldots,m\}$ and of length $\ell\in\{1,2,\ldots,m\}$, we define the \emph{successor of $\tilde{f}$} as follows: If all entries of $f$ are equal to $m$, then the successor of $f$ is defined as $\emptyset$ (a dummy value). Otherwise, in order to obtain the successor of $\tilde{f}$ from $\tilde{f}$, keep removing the final entry until it is distinct from $m$, and then raise the final entry by $1$. For example, if $m=5$, then the successor of $(1,3,5,5)$ is $(1,4)$.

We now describe an algorithm which determines whether $\app_{\F}(X)\geq k$ for some given $k\in\{0,\ldots,m\}$. Initialize the variable $h$ as $(1)$. Then, as long as $h$ is not $\emptyset$, do the following: Writing $h=(h_1,\ldots,h_{\ell})$, check if, for some $f\in\F$, the number of $i\in\{1,\ldots,\ell\}$ such that $x_{h_i}=f(x_i)$ is at least $k$. If so, then replace $h$ by its successor. If not, then there are two cases to distinguish:
\begin{itemize}
\item If $\ell=m$, then $h$ encodes a function $g:X\rightarrow X$ such that $\app_{\F}(g)<k$, so that $\app_{\F}(X)<k$. We thus output \enquote{false} and halt.
\item If $\ell<m$, then extend $h$ by adding an entry $1$ at the end.
\end{itemize}
If this loop is exited because the condition $h\not=\emptyset$ is no longer satisfied, then halt and output \enquote{true}.

This algorithm always halts with the correct answer to the question of whether $\app_{\F}(X)\geq k$. Indeed, note that the following hold:
\begin{enumerate}
\item At the beginning of each iteration of the loop, if $h\not=\emptyset$, we have that $\app_{\F}(g)\geq k$ for all functions $g$ on $X$ whose restriction to $\{x_1,x_2,\ldots,x_{\ell}\}$ is encoded by an $\ell$-tuple from $\{1,\ldots,m\}^{\ell}$ which is lexicographically smaller than $h=(h_1,\ldots,h_{\ell})$.
\item At the end of an iteration of the loop, if we reach the command to replace $h$ by its successor (and begin the next iteration of the loop), we have that $\app_{\F}(g)\geq k$ for all functions $g$ on $X$ whose restriction to $\{x_1,\ldots,x_{\ell}\}$ is encoded by $h=(h_1,\ldots,h_{\ell})$.
\end{enumerate}
Together, this implies that \emph{if} the algorithm halts, it will output the correct answer. Indeed, this is clear if the halting is with output \enquote{false}. Moreover, if the halting is with output \enquote{true}, then this is because the loop was exited with $h=\emptyset$, which means that at the previous loop iteration, $h$ was of the form $(m,m,\ldots,m)$, of length $\ell$, say. But every $\ell$-tuple with entries in $\{1,2,\ldots,m\}$ is lexicographically smaller or equal to the $\ell$-tuple $(m,m,\ldots,m)$, and so in this situation, every function $g$ on $X$ must satisfy $\app_{\F}(g)\geq k$, so that the output \enquote{true} reflects the correct answer.

It remains to show that the algorithm always halts. But this is clear because after each iteration of the loop, one of the following happens:
\begin{itemize}
\item The algorithm halts with output \enquote{false};
\item $h$ is replaced by $\emptyset$ (so that the loop will be exited at the beginning of the next iteration and the algorithm will halt with output \enquote{true}); or
\item $h$ is replaced by a lexicographically larger tuple from the finite set $\bigcup_{\ell=1}^m{\{1,\ldots,m\}^{\ell}}$.
\end{itemize}

From this decision algorithm for $\app_{\F}(X)\geq k$, it is straightforward to obtain an algorithm for computing the precise value of $\app_{\F}(X)$, as the largest value of $k\in\{0,\ldots,m\}$ such that $\app_{\F}(X)\geq k$.

Now, in our setting, we have that $X$ is a finite group $G$, and the repetition-free ordering $g_1,\ldots,g_m$ of the elements of $G$ is the one obtained by evaluating the command

{\tt GeneratorsOfDomain(DomainByGenerators(G))}.

A complete list of all endomorphisms of $G$ can be obtained using GAP's built-in command {\tt AllEndomorphisms(G)}, and from this, one can also compute a list of all affine maps of $G$.

We can thus apply the above described backtracking algorithm in the special case $X=G$ and $\F\in\{\End(G),\Aff(G)\}$, with the slight modification that in case $\F=\Aff(G)$, the loop is also exited if the first entry of $h$ is greater than $1$ (this is because by Lemma \ref{reductionLem}, if any function $f$ on $G$ with $\affapp(f)<k$ exists, it can be chosen such that additionally, $f(g_0)=g_1$ for any given elements $g_0,g_1\in G$). We note that the corresponding GAP code written by the author is also available, alongside a short PDF documentation, from his homepage under \url{https://alexanderbors.wordpress.com/sourcecode/approx/}. This allows us to verify the following:

\begin{proposition}\label{smallProp2}
The values of $\enapp(G)$ and $\affapp(G)$ for finite groups $G$ with $|G|\leq15$ are as specified in Table \ref{table4} below, which lists the groups according to their identifier in the GAP Small Groups Library \cite{EBB18a}.
\end{proposition}

\begin{table}[H]
\caption{Values of $\enapp(G)$ and $\affapp(G)$ for $|G|\leq15$.}
\label{table4}
\begin{center}
\begin{tabular}{|c|c|c|c|}
\hline
SmallGroups ID of $G$ & Structure description of $G$ & $\enapp(G)$ & $\affapp(G)$ \\ \hline
$(1,1)$ & $\{1\}$ & $1$ & $1$ \\ \hline
$(2,1)$ & $\IZ/2\IZ$ & $1$ & $2$ \\ \hline
$(3,1)$ & $\IZ/3\IZ$ & $1$ & $2$ \\ \hline
$(4,1)$ & $\IZ/4\IZ$ & $1$ & $2$ \\ \hline
$(4,2)$ & $(\IZ/2\IZ)^2$ & $2$ & $3$ \\ \hline
$(5,1)$ & $\IZ/5\IZ$ & $1$ & $2$ \\ \hline
$(6,1)$ & $\D_6$ & $0$ & $2$ \\ \hline
$(6,2)$ & $\IZ/6\IZ$ & $1$ & $2$ \\ \hline
$(7,1)$ & $\IZ/7\IZ$ & $1$ & $2$ \\ \hline
$(8,1)$ & $\IZ/8\IZ$ & $1$ & $2$ \\ \hline
$(8,2)$ & $\IZ/4\IZ \times \IZ/2\IZ$ & $1$ & $2$ \\ \hline
$(8,3)$ & $\D_8$ & $1$ & $2$ \\ \hline
$(8,4)$ & $\Dic_8$ & $1$ & $3$ \\ \hline
$(8,5)$ & $(\IZ/2\IZ)^3$ & $3$ & $4$ \\ \hline
$(9,1)$ & $\IZ/9\IZ$ & $1$ & $2$ \\ \hline
$(9,2)$ & $(\IZ/3\IZ)^2$ & $2$ & $3$ \\ \hline
$(10,1)$ & $\D_{10}$ & $0$ & $2$ \\ \hline
$(10,2)$ & $\IZ/10\IZ$ & $1$ & $2$ \\ \hline
$(11,1)$ & $\IZ/11\IZ$ & $1$ & $2$ \\ \hline
$(12,1)$ & $\Dic_{12}$ & $0$ & $2$ \\ \hline
$(12,2)$ & $\IZ/12\IZ$ & $1$ & $2$ \\ \hline
$(12,3)$ & $\Alt(4)$ & $0$ & $3$ \\ \hline
$(12,4)$ & $\D_{12}$ & $1$ & $3$ \\ \hline
$(12,5)$ & $\IZ/6\IZ\times\IZ/2\IZ$ & $1$ & $3$ \\ \hline
$(13,1)$ & $\IZ/13\IZ$ & $1$ & $2$ \\ \hline
$(14,1)$ & $\D_{14}$ & $0$ & $2$ \\ \hline
$(14,2)$ & $\IZ/14\IZ$ & $1$ & $2$ \\ \hline
$(15,1)$ & $\IZ/15\IZ$ & $1$ & $2$ \\ \hline
\end{tabular}
\end{center}
\end{table}

\section{Concluding remarks}\label{sec6}

We conclude this paper with some open questions and problems for further research.

Note that while we gained a deeper understanding of the asymptotic behavior of $\enapp(G)$ and $\affapp(G)$ as $|G|\to\infty$, determining the \emph{precise} values of the two functions on a given finite group remains a challenging problem. Nonetheless, it would be nice to know these precise values at least on a few \enquote{basic} classes of groups. For example, consider the following question, which is, to the author's knowledge, open in this generality:

\begin{question}\label{ques1}
Is $\affapp(\IZ/n\IZ)=2$ for all integers $n\geq2$?
\end{question}

By Lemma \ref{cyclicLem}(2,3), we do know that $\affapp(\IZ/n\IZ)=2$ when $n$ is prime or a power of $2$, and it also holds for $n\leq15$ by Proposition \ref{smallProp2}.

In this context, it would also be nice to extend our list of $\enapp(G)$ and $\affapp(G)$ for \enquote{small} $G$ from Proposition \ref{smallProp2} further, which might also lead to some more interesting conjectures about their behavior on certain classes of finite groups:

\begin{problem}\label{prob}
Determine the precise values of $\enapp(G)$ and $\affapp(G)$ for all finite groups $G$ of order up to $N$, for $N\in\IN$ as large as possible.
\end{problem}

Finally, it would be interesting to determine provably asymptotically best possible upper bounds on $\enapp(G)$ and $\affapp(G)$ in general. With regard to this, we note the following: If a finite group $G$ has a universal $\ell$-tuple, then $|G|^{\log_2{|G|}}\geq|\End(G)|\geq|G|^{\ell}$, so that the \emph{lower} bounds on $\enapp(G)$ and $\affapp(G)$ which we can prove with our current methods from Section \ref{sec2} are at best logarithmic in $|G|$. This leads to the question whether we hit this boundary for a good reason:

\begin{question}\label{ques2}
Is $\enapp(G)\leq\log_2{|G|}$ and $\affapp(G)\leq1+\log_2{|G|}$ for all nontrivial finite groups $G$? If not, is it at least the case that $\affapp(G)=\O(\log{|G|})$ as $|G|\to\infty$ for finite groups $G$?
\end{question}


\begin{thebibliography}{1}

\bibitem{Bor14a}
A. Bors,
On the endomorphism monoids of some groups with abelian automorphism group,
preprint (2014), \href{https://arxiv.org/abs/1411.4190}{arXiv:1411.4190 [math.GR]}.

\bibitem{Bor16a}
A. Bors,
Classification of Finite Group Automorphisms with a Large Cycle,
\emph{Comm. Algebra} \textbf{44}(11):4823--4843, 2016.

\bibitem{Bor16b}
A. Bors,
Finite groups with an automorphism inverting, squaring or cubing a non-negligible fraction of elements,
\emph{J.~Algebra Appl.}, \textbf{18}(3):30 pp., 2019.

\bibitem{CT02a}
C.~Carlet and Y.~Tarannikov,
Covering Sequences of Boolean Functions and Their Cryptographic Significance,
\emph{Des. Codes Cryptogr.} \textbf{25}(3):263--279, 2002.

\bibitem{EBB18a}
B.~Eick, H.U.~Besche and E.~O'Brien,
SmallGrp -- The GAP Small Groups Library,
version 1.3 (9 April 2018), \url{https://www.gap-system.org/Manuals/pkg/SmallGrp-1.3/doc/chap0.html}.

\bibitem{FX18a}
Y.~Fan and B.~Xu,
Fourier transforms and bent functions on finite groups,
\emph{Des.~Codes Cryptogr.} \textbf{86}(9):2091--2113, 2018.

\bibitem{GAP4}
The GAP~Group,
\emph{GAP -- Groups, Algorithms, and Programming, Version 4.10.2} (2019),
\url{http://www.gap-system.org}.

\bibitem{JK75a}
D.~Jonah and M.~Konvisser,
Some non-abelian $p$-groups with abelian automorphism groups,
\emph{Arch.~Math.~(Basel)} \textbf{26}(1):131--133, 1975.

\bibitem{JS75a}
D.~Jonah and B.M.~Schreiber,
Transitive affine transformations on groups,
\emph{Pacific J.~Math.} \textbf{58}(2):483--509, 1975.

\bibitem{KY18a}
R.D.~Kitture and M.K.~Yadav,
Finite Groups with Abelian Automorphism Groups: A Survey,
in: N.S.N.~Sastry and M.K.~Yadav (eds.), \emph{Group Theory and Computation},
Springer (Indian Statistical Institute Series), Singapore, 2018, pp.~119--140.

\bibitem{LS12a}
M. Larsen and A. Shalev,
Fibers of word maps and some applications,
\emph{J. Algebra} \textbf{354}:36--48, 2012.

\bibitem{Poi06a}
L.~Poinsot,
Bent functions on a finite nonabelian group,
\emph{J.~Discrete Math.~Sci.~Cryptogr.} \textbf{9}(2):349--364, 2006.

\bibitem{Poi12a}
L.~Poinsot,
Non Abelian bent functions,
\emph{Cryptogr.~Commun.} \textbf{4}:1--23, 2012.

\bibitem{PP11a}
L.~Poinsot and A.~Pott,
Non-Boolean almost perfect nonlinear functions on non-Abelian groups,
\emph{Internat.~J.~Found.~Comput.~Sci.} \textbf{22}(6):1351--1367, 2011.

\bibitem{Rob96a}
D.J.S.~Robinson,
\emph{A Course in the Theory of Groups},
Springer (Graduate Texts in Mathematics, 80), New York, 2nd.~edn.~1996.

\bibitem{Tao10a}
T. Tao,
\emph{254A, Notes 0a: Stirling's formula},
online notes, \url{https://terrytao.wordpress.com/2010/01/02/254a-notes-0a-stirlings-formula/}.

\bibitem{Win72a}
D.L.~Winter,
The automorphism group of an extraspecial $p$-group,
\emph{Rocky Mountain J.~Math.} \textbf{2}(2):159--168, 1972.

\bibitem{Xu15a}
B.~Xu,
Bentness and nonlinearity of functions on finite groups,
\emph{Des.~Codes Cryptogr.} \textbf{76}(3):409--430, 2015.

\end{thebibliography}
\end{document}